\documentclass[twoside,a4paper]{amsart}
\usepackage{amsmath,amssymb,latexsym,amsthm}
\usepackage{enumerate}
\usepackage{color}
\usepackage[svgnames]{xcolor}
\usepackage{paralist}
\usepackage[multiple]{footmisc}
\usepackage[nodisplayskipstretch]{setspace}
\usepackage{fancyhdr}
\usepackage{centernot}
\usepackage{tikz}
  \usetikzlibrary{matrix}
  \usetikzlibrary{calc}
\usepackage{marginnote}
\usepackage{chngpage}
\usepackage{multirow, array}
\usepackage{float}

\theoremstyle{plain}
\newtheorem{thm}{Theorem}[section]
\newtheorem{lem}[thm]{Lemma}
\newtheorem{propn}[thm]{Proposition}
\newtheorem{cor}[thm]{Corollary}

\newtheorem{example}[]{Example}

\newtheorem{clm}[thm]{Claim}

\theoremstyle{definition}
\newtheorem{defn}[thm]{Definition}
\newtheorem{rem}[thm]{Remark}

\newcommand{\La}{\mathcal{L}}
\newcommand{\M}{\mathcal{M}}

\newcommand{\U}{\ensuremath \mathcal{U}}

\newcommand{\Z}{\mathbb{Z}}
\newcommand{\Prime}{\mathbb{P}}

\newcommand{\vphi}{\varphi}
\newcommand{\pint}{\mathbb{Z}_p}
\newcommand{\tp}{\mathrm{tp}}

\newcommand{\iso}{\cong}
\newcommand{\Qp}{\ensuremath \mathbb{Q}_p}

\newcommand{\indep}[1][]{%
  \mathrel{
    \mathop{
      \vcenter{
        \hbox{\oalign{\noalign{\kern-.3ex}\hfil$\vert$\hfil\cr
              \noalign{\kern-.7ex}
              $\smile$\cr\noalign{\kern-.3ex}}}
      }
    }\displaylimits_{#1}
  }
}

\def\vvdots{\vbox{\lineskiplimit=0pt 
\kern6pt \hbox{.}\hbox{.}\hbox{.}}}

\begin{document}
\title{Generalised stability of ultraproducts of finite residue rings}
\author{Ricardo Isaac Bello Aguirre}
\date{}
\begin{abstract}
We study ultraproducts of finite residue rings $\underset{n\in\mathbb{N}}{\prod}\Z/n\Z \diagup \mathcal{U} $ where $\mathcal{U}$ is a non-principal ultrafilter.
We find sufficient conditions of the ultrafilter $\mathcal{U}$ to determine if the resulting ultraproduct $\underset{n\in\mathbb{N}}{\prod}\Z/n\Z \diagup \mathcal{U}$ has simple, NIP, NTP2 but not simple nor NIP, or TP2 theory, noting that all these four cases occur.

\end{abstract}

\maketitle

\section{Introduction}
It is well known that any pseudofinite field has supersimple rank 1 theory. This follows e.g. from the main theorem and Proposition 4.11 of \cite{CDM}.
In this paper we investigate generalised stability properties of arbitrary pseudofinite residue rings. More specifically, we describe non-principal ultrafilters on $\mathbb{N}$ such that the ultraproduct $\prod \Z/n\Z \diagup \mathcal{U}$ is supersimple, or NIP but non-simple, or NTP2 but not NIP or simple, or TP2, noting that all these possibilities occur. The information is depicted in the following diagram.\newline

\begin{center}
\begin{tabular}{ | l | l | p{6cm}| }
\hline
\multirow{3}{2.7cm}{Bounded number of prime divisors} & \multicolumn{2}{l|}{Bounded exponents, (e.g. $\underset{p\in\Prime}{\prod} \Z / p^b \Z \diagup \mathcal{U}$, for fixed $b$) Simple case}\\\cline{2-3}
%\hline
                                                & \multirow{2}{2cm}{Unbounded exponents} & {Bounded set of primes, \newline (e.g. $\underset{n\in\mathbb{N} }{\prod} \Z/q^n\Z \diagup \mathcal{U}$, for fixed prime $q$) NIP case} \\ \cline{3-3}
                                               &  &  Unbounded set of primes, \newline(e.g. $\underset{p\in\Prime, n\in\omega}{\prod} \Z/p^n\Z \diagup \mathcal{U}$) NTP2 case \\
                                               \hline
\multicolumn{3}{|l|}{Unbounded number of prime divisors, (e.g. $\underset{n\in\omega}{\prod} \Z/n\Z \diagup \mathcal{U}$)  TP2 case}  \\
\hline
\end{tabular}
\end{center}

We now state the main theorems that make up this paper and indicate where can they be found inside the paper.

\begin{thm}[Corollary \ref{NIPcase}]\label{theoremnip}
Let $\mathcal{U}$ be a non-principal ultrafilter on $\mathbb{N}$ and let $U\in \mathcal{U}$ and $b\in \mathbb{N}$ be such that every $n\in U$ is a product of powers of fewer than $b$ primes each prime being less than $b$. Then $\underset{n\in\mathbb{N}}{\prod} \Z/ n\Z\diagup \mathcal{U} $ is NIP.
\end{thm}

\begin{thm}[Corollary \ref{NTP2case}]\label{theoremntp2}
Let $\U$ be an ultrafilter on $\mathbb{N}$ such that there exists $b\in\mathbb{N}$ and $U\in\U$ such that for  every $n\in U$ there are at most $b$ prime divisors. Then $\underset{n\in\mathbb{N}}{\prod} \Z/n\Z\diagup \mathcal{U}$ is NTP2.

\end{thm}

\begin{thm}[Corollary \ref{simplecase}]\label{theoremsimple}
Consider the collection of rings of the form $\Z /n \Z$ with $n\in\mathbb{N}$. Consider also an ultrafilter on $\mathbb{N}$ such that there exists $b\in\mathbb{N}$ and $ U\in \mathcal{U}$ such that if $n\in U$ then $n$ is a product of fewer than $b$ primes and if $p^l$ divides $n$ then $l<b$. Then $\underset{n\in\mathbb{N}}{\prod} \Z/n\Z \diagup \mathcal{U}$ is supersimple.
\end{thm}

\begin{thm}[Corollary \ref{TP2case}]\label{theoremtp2}
Let $\mathcal{U}$ be an ultrafilter on $\mathbb{N}$ such that for every $b\in\mathbb{N}$ the set $$\{n\in\mathbb{N}:\ there\ are\ at\ least\ b\ distinct\ prime\ divisors\ of\ n\}$$ is in $\mathcal{U}$. Then $\underset{n\in\mathbb{N} }{\prod }\Z/n\Z \diagup \mathcal{U}$ has TP2 theory.
\end{thm}

In the second section we present both Theorem \ref{theoremnip} and Theorem \ref{theoremntp2}, concerning the NIP and NTP2 cases respectively.

In the third section we present Theorem \ref{theoremsimple} about the simple theory case, and also mention some results on coordinatisation, as found in \cite{HKP}.

Finally in the fourth section we present Theorem \ref{theoremtp2} for the TP2 case.

We heavily use definability of certain pseudofinite residue rings in well known valued fields. When handling valued fields, we will usually work in the multisorted Denef-Pas language of valued fields $\La_{valf}$ defined with more detail in Definition \ref{Denefpas}. 

We will write $\mathbb{P}$ for the set of prime numbers, $\mathcal{U}$, $\mathcal{V}$, $\mathcal{W}$, $\ldots$ for ultrafilters, and $U$, $V$, $W$, $\ldots$ for elements of ultrafilters. Also we write $x$, $y$, $\ldots$ for variables and  $a$, $b$, $\ldots$ for parameters. If we want to emphasise that a variable or parameter is a tuple we will write $\bar{x}$, $\bar{y}$ or $\bar{a}$, $\bar{b}$ respectively. 
When $\mathcal{U}$ is an ultrafilter on $I$ we often write $[a_i]_U$ for an element of $\underset{i\in I}{\prod} R_i /\mathcal{U}$; we may write $[a_i]$ if the ultrafilter $\mathcal{U}$ is clear from the context. Unless specified otherwise, ``$\mathcal{U}$ is an ultrafilter'' means $\mathcal{U}$ is a non-principal ultrafilter. For $V$ an element in $\mathcal{U}$ a (non-principal) ultrafilter we call $\{V\cap U: U\in\mathcal{U} \}$ the {\textit{induced ultrafilter on $V$ by $\mathcal{U}$}}.
We use $T$ to denote a complete first order theory and $\overline{\M}$ to denote a model $T$.

We present now the definition of the tree property of the second kind.

\begin{defn}\label{DefNTP2}
We say that an $\mathcal{L}$-formula $\vphi(x,y)$ has the \textit{tree property of the second kind}, in short TP2, if there are $\{ b_{i,j}: i,j<\omega\}$ in $\overline{\M}$ and $l<\omega$ such that:
\begin{enumerate}[i)]

\item The set $\{ \vphi(x,b_{i,j}):j\in \omega\}$ is $l$-inconsistent, for all $i\in\omega$;
\item For all $\xi\in\omega^\omega$ the set $\{ \vphi(x,b_{i,\xi(i)}): i\in \omega\}$ is consistent.

\end{enumerate}
We say a theory {\textit{$T$ is NTP2}}, or does not have TP2, if no formula satisfies TP2.
\end{defn}

By Definition 3.1 of \cite{CHE} the class of TP2 theories is the same as the class given by Definition \ref{DefNTP2} with $l=2$ in i).

\begin{example}
The theory of the generic ultrahomogeneous triangle-free graph has the tree property of the second kind.

\end{example}

A proof is given in Example 3.13 of \cite{CHE}.

\begin{defn}
An $\mathcal{L}$-formula $\vphi(x,y)$ has the \textit{independence property} for $T$ if there are $(a_i)_{i\in\omega}$ and $(b_I)_{I\subseteq \omega}$ in $\overline{\M}$ such that $\overline{\M}\models \vphi(a_i,b_I)$ if and only if $i\in I$.

We say that a theory is \textit{NIP} if no formula satisfies the independence property.
\end{defn}

\begin{defn}
We say that a formula $\psi(x,y)$ has the \textit{tree property} for $T$ if there are $(a_{\eta})_{\eta\in\omega^{< \omega}}$ and some $k\geq 2$ such that:
\begin{enumerate}[a)]
\item For every $\delta\in\omega^{<\omega}$ the set of formulas $\{\psi(x,a_{\delta\frown l}): l \in \omega \}$ is $k$-inconsistent;
\item If $\eta\in\omega^\omega$ then the set $\{\psi(x, a_{\eta\upharpoonright n}): n\in\omega \}$ is consistent.
\end{enumerate}
A theory is said to be \textit{simple} if no formula has the tree property.
\end{defn}

The class of NTP2 theories is a simultaneous generalisation of both the classes of NIP and simple theories.
\begin{rem}
Any simple theory is NTP2. Any NIP theory is NTP2.
\end{rem}

By Example 7.7 of \cite{CHE}, any ultraproduct $\underset{p\in\Prime}{\prod} \Qp /\mathcal{U} $ of $p$-adic fields, where $\mathcal{U}$ is a non-principal ultrafilter on $\Prime$, has NTP2 theory.
This follows from the more general AKE-like result found in Theorem 7.6 in \cite{CHE}. Paraphrased, this says the following.

\begin{propn} [\cite{CHE} Theorem 7.6]\label{Chernikov}
Let $\overline{K}=(K,\Gamma,k,v:K\rightarrow\Gamma,ac:K\rightarrow k)$ be a Henselian valued field of characteristic $(0,0)$ in the Denef-Pas language. Then the height $\lambda$ of an array of parameters $\{\bar{b}_{i,j}: j<\omega, i<\lambda \} $ in $K$ which satisfies clauses i) and ii) from Definition \ref{DefNTP2} is less than the height of arrays of parameters either in $k$ or in $\Gamma$ which satisfies clauses i) and ii) from Definition \ref{DefNTP2}.

In particular we have that if $k$ has NTP2 theory in $\La_{rng}$, then so does $\overline{K}$.
\end{propn}

Observe that in Proposition \ref{Chernikov}, if we consider $\overline{K}:=(\prod\mathbb{Q}_p/\mathcal{U},\Gamma, k, v, ac)$ then $\overline{K}$ is strictly NTP2, in the sense that since $k$ is a pseudofinite field then it has IP and so $\overline{K}$ has IP, and $\Gamma$ has SOP so $\overline{K}$ has SOP.

\noindent\textbf{Acknowledgements.}
I would like to thank greatly Dugald Macpherson for his supervision during this project, and Anand Pillay for his very helpful comments and suggestions around this work, specially around using coordinatisability in section 3. I would like to also thank CONACYT for the financial support that allowed this work.

\section{NIP, and NTP2 cases}

First, for a fixed prime $p$ we consider the ring $\underset{n\in\mathbb{N} }{\prod} \Z/p^n\Z \diagup \mathcal{U}$.

\begin{propn}\label{nip}
Fix a prime $p$. Ultraproducts of the form $\underset{n\in\mathbb{N} }{\prod} \Z/p^n\Z \diagup \mathcal{U}$ with $\mathcal{U}$ a non-principal ultrafilter on $\mathbb{N}$ are interpretable in $\underset{n\in\mathbb{N}}{\prod}\Qp \diagup \mathcal{U}$ and hence are NIP.
\end{propn}
\begin{proof} We first recall that $\mathbb{Q}_p$ has NIP theory, c.f. \cite{BEL}, \cite{DEL} or \cite{MAT}. We will show that uniformly in $n$, $\Z / p^n\Z$ is interpretable in $\mathbb{Q}_p$, the $p$-adic numbers. We know that the valuation ring, $\Z_p$ is definable inside the valued field $\mathbb{Q}_p$. Also we can use a parameter $a\in\Qp$ with $v(a)=n$ to define $p^n\Z_p$, since $p^n\Z_p=\{x\in\Z_p:v(x)\geq v(a)\}$. Hence the structure $\Z_p/p^n\Z_p\iso \Z/p^nZ$ is interpretable in $\mathbb{Q}_p$, uniformly in $n$ (a parameter varying through $\Z$). Furthermore the ultraproduct $\underset{n\in\mathbb{N} }{\prod} \Z/p^n\Z \diagup \mathcal{U}$ is interpretable by the same formula in the ultrapower $\underset{n\in\mathbb{N}}{\prod}\Qp \diagup \mathcal{U}$ which is still NIP. Since being NIP is preserved under interpretability we conclude that ultraproducts of the form $\underset{n\in\mathbb{N} }{\prod} \Z/p^n\Z \diagup \mathcal{U}$ are NIP.
\end{proof}

We present now a lemma that will be useful further on. Here for $j$ in an index set $J$, and a collection of structures $(A_j)_{j\in J}$ then $\pi_j$ is the usual projection map from $\underset{k\in J}{\prod} A_k$ to $A_j$. We extend this notation to ultrafilters, i.e. if we consider $\mathcal{U}$ an ultrafilter on $\prod I_k$  we will denote by $\pi_j(\mathcal{U})$ the ultrafilter $\{V\subseteq I_j: \exists U\in\mathcal{U}( \pi_j(U)= V ) \}$ on $I_j$.

\begin{lem}\label{CRT}
Let $\{I_k\}_{k<n}$ be a family of index sets, for each $I_k$ let $\{R^k_{i} \}_{i\in I_k}$ be a family of rings indexed by $I_k$, and $\mathcal{U}$ an ultrafilter on $\underset{k<n}{\prod} I_k$.
 Then $$\underset{(i_0,\ldots,i_{n-1})}{\prod} (R^0_{i_0}\times\cdots\times R^{n-1}_{i_{n-1}})\diagup \mathcal{U}\iso (\underset{i_0}{\prod}R^0_{i_0}\diagup \pi_0(\mathcal{U})) \times\cdots\times (\underset{i_{n-1}}{\prod}R^{n-1}_{i_{n-1}}\diagup \pi_{n-1}(\mathcal{U}) ).$$
\end{lem}
\begin{proof}
We can show that the assignment $\vphi$ given by sending $[(a)_{(k_0,\ldots,k_{n-1})} ]_\mathcal{U}$ 
to $( [(a_{k_0})]_{\pi_0(\mathcal{U})}, \ldots, [(a_{k_{n-1}})]_{\pi_{n-1}(\mathcal{U})})$ is an isomorphism.\\
\end{proof}

\begin{cor}\label{NIPcase}
Let $\mathcal{U}$ be a non-principal ultrafilter on $\mathbb{N}$ and let $U\in \mathcal{U}$ and $b\in \mathbb{N}$ be such that every $n\in U$ is a product of powers of fewer than $b$ primes each prime being less than $b$. Then $\underset{n\in\mathbb{N}}{\prod} \Z/ n\Z\diagup \mathcal{U} $ is NIP.
\end{cor}
\begin{proof}
Put $R':=\underset{n\in\mathbb{N}}{\prod} \Z/ n\Z\diagup \mathcal{U} $. Let $U\in\mathcal{U}$ be as in the hypothesis,  and consider the induced ultrafilter on $U$ by $\mathcal{U}$, namely $\mathcal{V}:=\{ U \cap V: V \in\mathcal{U} \}$. Put $R:=\underset{n\in U}{\prod} \Z/n\Z \diagup \mathcal{V}$. Then $R\iso R'$.

Furthermore we can find $V\in\mathcal{V}$ such that every $n\in V$ has the same $d$ prime factors, $p_1,p_2,\ldots,  p_d$. Considering $\mathcal{W}$ the induced ultrafilter on $V$ by $\mathcal{V}$ and using Lemma \ref{CRT} we have that $R\iso (\underset{m}{\prod} (\Z/p_1^{m}\Z)\diagup \mathcal{W}_1)\times\cdots\times (\underset{m}{\prod}(\Z/p_d^m\Z)\diagup\mathcal{W}_d)$, for some ultrafilters $\mathcal{W}_1, \ldots,\mathcal{W}_d $ on $\mathbb{N}$. Since each of $\underset{m}{\prod} (\Z/p_k^m\Z)\diagup\mathcal{W}_k$ is NIP by Proposition \ref{nip} we can conclude that $R'$ is NIP.
\end{proof}
We take a moment here to note that we are using and will use the following result.
\begin{propn}\label{}
Each of $R_1,\ldots,R_n$ rings is $NIP$ (respectively, $simple$, $NTP2$) if and only if $R_1 \times\ldots\times R_n$ is $NIP$ (respectively ,simple, $NTP2$).
\end{propn}

This follows from the following two lemmas.
\begin{lem}\label{Daniel} Let $\La:=\La_1\sqcup\La_2$. Let $\vphi(\bar{x},\bar{y})\in\La$, where $\bar{x}\in\La_1$ and $\bar{y}\in\La_2$. Then $\vphi(\bar{x},\bar{y})$ is equivalent to a finite disjunction of formulas of the form $\theta(\bar{x}) \wedge \psi(\bar{y)} $, where $\theta(\bar{x})\in\La_1$ and $\psi(\bar{y})\in\La_2$.                   \end{lem}
\begin{proof}
By induction on the length of the formula $\vphi$. See for example Exercise 9.6.15 of \cite{HOD}.
\end{proof}

\begin{lem}
Let $\La_1,\ldots, \La_n$ be disjoint languages. If $A_1,\ldots,A_n$ are $NIP$, $simple$ or $NTP2$ $\La_i$-structures then so is $A_1\sqcup\ldots\sqcup A_n$.
\end{lem}
To see this, it suffices to show that each of $NIP$, $simple$ and $NTP2$ are preserved under binary disjoint unions. We present as an example the NTP2 case.

\begin{rem}
Let $A$ be $\La$-structure and $B$ be an $\La'$-structure, where $\La$ and $\La'$ are disjoint. Then the $\La \sqcup \La'$-structure $A\sqcup B$ is NTP2 if and only if both $A$ and $B$ are NTP2.
\end{rem}
\begin{proof}
Note that if either $A$ or $B$ have TP2 then $A\sqcup B$ the disjoint union has TP2 witnessed by the same formula and array.
Now assume that $A\sqcup B $ has TP2. Then TP2 is witness by a formula $\vphi(x,\bar{y})$ with $x$ a single variable and an array $(\bar{b}_{i,j})$. But by Lemma \ref{Daniel} $\vphi$ is a finite disjunction of formulas of the form $\theta_l \wedge \psi_l$. So for a particular index $k$, there is a subarray from $(\bar{b}_{i,j})$ that witness the tree property of the second kind for $\theta_k \wedge \psi_k$. Furthermore we can find an array of parameters only in $A$ such that $\theta_k(x,\bar{y})$ holds, or only in $B$ such that $\psi_k(x,\bar{y})$ holds. This means that either $A$ or $B$ have the tree property of the second kind.
\end{proof}

Now let both $p\in\Prime$, and $n\in\mathbb{N}$ vary in $\underset{(p,n)\in\Prime\times\omega}{\prod} \Z/p^n\Z \diagup \mathcal{U}$ with $\mathcal{U}$ a non-principal ultrafilter on $\Prime\times\omega$.
Consider the following class of residue rings $\mathcal{C}:=\{\mathbb{Z}/p^n\mathbb{Z}:p\in\mathbb{P},n\in \omega\}$.

\begin{propn}\label{ntp2}
Any ultraproduct of rings in $\mathcal{C}=\{\mathbb{Z}/p^n\mathbb{Z}:p\in\mathbb{P},n\in \omega\}$ has NTP2 theory.
\end{propn}
\begin{proof}
We first note that $\mathbb{Z}/p^n\mathbb{Z} \cong \pint/ p^n\pint$, where $\mathbb{Z}_p$ denotes to the ring of $p$-adic integers.

Let $\U$ be a non-principal ultrafilter on $\Prime\times \omega$. For each $(p,e)\in\Prime\times \omega$ choose an element $a_{p^e} \in \Qp$ such that $v(a_{p^e})=e$. This defines $p^e\Z_p$ in $\mathbb{Q}_p$, as the set of elements in $\Z_p$ with value greater or equal to the value of $a_{p^e}$. Hence $\Z_p/p^e\Z_p$ is interpretable in $(\Qp,a_{p^e})$, and $R:=\underset{(p,e)\in\Prime\times\omega}{\prod}(\Z_p/p^e\Z_p)\diagup \U$ is interpretable in $\underset{(p,e)\in\Prime\times\omega}{\prod} (\Qp, \overline{a_{p^e}})\diagup\U$, where $\overline{a_{p^e}}$ is an element in $\underset{(p,e)\in\Prime\times\omega}{\prod}\Qp$ with $(p,e)$-projections equal to $a_{p^e}$. Since by Proposition \ref{Chernikov} we have that $\underset{p\in\Prime}{\prod}\Qp/\mathcal{W}$ is NTP2 then $\underset{(p,e)\in\Prime\times\omega}{\prod}(\Qp, a_{p^e}) $ is also NTP2.

By the above observations, $R$ is NTP2. So every ultraproduct in $\mathcal{C}$ has NTP2 theory.
\end{proof}

Furthermore, in Proposition \ref{ntp2} $R$ need not be simple or NIP. If the ultrafilter concentrates on a prime $p$ then $R$ is NIP but not simple, and if it concentrates on prime powers with exponent 1 then $R$ is supersimple since the ultraproduct is then a pseudofinite field. Since pseudofinite fields have the independence property, $R$ is not $NIP$.

\begin{cor}\label{NTP2case}
Let $\U$ be an ultrafilter on $\mathbb{N}$ such that there exist $b\in\mathbb{N}$ and $U\in\U$ such that every $n\in U$ has at most $b$ prime factors. Then $\underset{n}{\prod} \Z/n\Z \diagup \mathcal{U}$ is NTP2.
\end{cor}
\begin{proof}
Let $R'$ be such an ultraproduct and $U$ an element of the ultrafilter as in the hypothesis. Consider $\mathcal{V}:=\{V\cap U: V\in\U \}$. Put $R:=\underset{n\in U}{\prod} \Z/n\Z\diagup \mathcal{V}$. Furthermore there is a $V\in \mathcal{V}$ such that every $n\in V$ has exactly $d$ prime divisors. Choose as earlier an ultrafilter $\mathcal{W}$ on $V$ such that $R\iso \underset{n\in V}{\prod}( \Z/p_{n(1)}^{e_{n(1)}}\Z \times \cdots \times \Z/p_{n(d)}^{e_{n(d)}} \Z )\diagup \mathcal{W} $. Using Lemma \ref{CRT} we have that 
$R \iso (\underset{(p_{n(1)},e_{n(1)})}{\prod} \Z/p_{n(1)}^{e_{n(1)}}\Z \diagup \mathcal{W}_1) \times\cdots\times (\underset{(p_{n(d)},e_{n(d)})}{\prod} \Z/p_{n(d)}^{e_{n(d)}}\Z \diagup \mathcal{W}_d )$. Hence by Proposition \ref{ntp2} we have that $R$ (and therefore $R'$) has $NTP2$ theory.
\end{proof}

\section{Simple case}

Now fix $b\in\mathbb{N}$ and consider the following ultraproduct, $\underset{p\in\Prime}{\prod} \Z/p^b\Z \diagup \mathcal{U}$.

In [Che-Hru] we find the following definition.
\begin{defn}
Let $D \subseteq N$ be structures possibly in different languages with $D$ definable in $N$ , and let $a\in N^{eq}$ be a canonical parameter for $D$.
\begin{enumerate}	
\item $D$ is \textit{canonically embedded} in $N$ if the $0$-definable relations of $D$ are the relations on $D$ which are $a$-definable in the sense of $N$.
\item $D$ is \textit{stably embedded} in $N$ if every $N$-definable relation on $D$ is $D,a$-definable, uniformly, in the structure $N$.
\item $D$ is \textit{fully embedded} in $N$ if it is both canonically and stably embedded in $N$.
\end{enumerate}

\end{defn}

In \cite{CHE} it is mentioned in the proof of Theorem 7.6 that if $\bar{K}=(K,\Gamma, k, v,ac)$ is a henselian valued field of characteristic $(0,0)$ in the three sorted Denef-Pas language then $\Gamma$ and $k$ are stably embedded with no new induced structure so are fully embedded. For completeness we include a proof.
We first recall the definition of the Denef-Pas language for Henselian valued fields.

\begin{defn}\label{Denefpas}
The \textit{Denef-Pas language} is a three sorted language, with a sort for the valued field in the language of rings $\La_{rng}=\{\cdot,+,-,0,1\}$, a sort for the ordered abelian group, in the language of ordered abelian groups $\La_{ogps}=\{+,0,<\} $together with an extra symbol $\infty$, and a sort for the residue field in the language of rings, $\La_{rng}$. We also include symbols for a valuation map $v:K\rightarrow \Gamma$ and an angular component map $\bar{ac}: K\rightarrow k$.
\end{defn}

\begin{propn}\label{Chat}
Let $\bar{K}=(K,\Gamma, k, v,ac)$ be a Henselian valued field of characteristic $(0,0)$ in the Denef-Pas language. Then the value group $\Gamma$ and the residue field $k$ are fully embedded.
\end{propn}
\begin{proof}

In the Denef-Pas language we have elimination of field quantifiers, cf. \cite{PAS}, or \cite{PASJSL}.

Let us show first that $\Gamma$ is stably embedded. Consider a $\bar{K}$-definable relation $R$ on $\Gamma$, defined by $\varphi(\bar{x},\bar{\alpha},\bar{\beta})$. By Denef-Pas quantifier elimination we may assume $\vphi$ has the form  $\bar{Q}(\bar{a},\bar{b}) \psi(\bar{x},\bar{\alpha},\bar{a},\bar{\beta}, \bar{b})$ where $\psi$ is a quantifier free formula, $\bar{Q}$ is a tuple of quantifiers on the group and residue field sorts, $\bar{x}$ is a tuple of free variables of the valued field sort, $\bar{\alpha}$ and $\bar{a}$ are tuples of free variables and bound variables respectively of the ordered group sort, and $\bar{\beta}$ and $\bar{b}$ are tuples of free and bound variables respectively from the residue field sort. 

Furthermore we may assume $\psi$ is a disjunction of formulas of the form $\psi_1(\bar{x}) \wedge \psi_2(v(t_2(\bar{x}),\bar{\alpha},\bar{a}  ) \wedge \psi_3(\bar{ac}(t_3(\bar{x})),\bar{\beta} , \bar{b})$, where $\psi_1(\bar{x})$ is a formula without quantifiers on the valued field sort, $\psi_2 (v(t_2(\bar{x})),\bar{\alpha},\bar{a})$ is a formula without quantifiers on the value group sort, and $\psi_3 (\bar{ac}(t_3(\bar{x})),\bar{\beta},\bar{b}  )$ is a quantifier free formula from the residue field sort, also $t_2(\bar{x})$ and $t_3(\bar{x})$ are terms obtained from the variables $\bar{x}$ via the operations from the valued field sort. We may assume the variables from $\bar{a}$ only appear in formulas like $\psi_2$, and the variables from $\bar{b}$ only appear in formulas like $\psi_3$. Hence $\vphi$ is equivalent to a disjunction of formulas of the form $\psi_1(\bar{x}) \wedge \vphi_2(v(t_2(\bar{x})),\bar{\alpha})\wedge\vphi_3(\bar{ac}(t_3(\bar{x})), \bar{\beta})$. Here $\psi_1$ is a quantifier free formula on the sort of valued fields, $\vphi_2(v(t_2(\bar{x}),\bar{\alpha})$ is a (quantified) formula from the sort of ordered groups where the bound variables are among $\bar{a}$, and $\vphi_3(\bar{ac}(t_3(\bar{x})), \bar{\beta})$ is a (quantified) formula from the residue field sort where the bound variables are from $\bar{b}$. 

Since the formula $\vphi$ defines a relation on $\Gamma$ we end up with a formula made up with a disjunction of formulas of the form of $\vphi_2$ and the parameters involved are all from $\Gamma$, possibly of the form $v(\bar{p})$ for some $\bar{p}\in K$.

Next we show that $\Gamma$ is canonically embedded. Consider now $S$ an $\emptyset$-definable relation in $\Gamma$, defined by $\vphi(\bar{x},\bar{\alpha},\bar{\beta})$. By the above argument we end up with $S$ being definable by a disjunction of formulas of the form $\vphi_2(v(t_2(\bar{x})),\bar{\alpha})$ with no parameters. Hence $S$ is $\emptyset$-definable and so $\Gamma$ is canonically embedded in $\bar{K}$.

In an analogous way we have that when a formula defines a subset of $k$ the only part of the formulas in the disjunction of formulas of the form $\psi_1(\bar{x}) \wedge \vphi_2(v(t_2(\bar{x}),\bar{\alpha})\wedge\vphi_3 (\bar{ac}(t_3(\bar{x})),\bar{\beta})$ we are interested in is that corresponding to $\vphi_3(\bar{ac}(t_3(\bar{x})),\bar{\beta})$ and every parameter used can be taken to be from $k$, where some may be of the form $\bar{ac}(\bar{p})$ for $\bar{p} \in K$. Hence $k$ is stably embedded. Furthermore an $\emptyset$-definable relation in $k$ defined by a formula $\vphi(\bar{x},\bar{\alpha},\bar{\beta} )$ ends up being $\emptyset$-definable by a formula only in the residue field sort. Hence $k$ is canonically embedded in $\bar{K}$.

\end{proof}

In the next definition and proposition, taken from \cite{HKP}, we work in a saturated model $\overline{M}=\overline{M}^{eq}$ of $T=T^{eq}$.

\begin{defn}
$\ $ 
\begin{itemize}
\item Suppose that $\mathcal{P}$ is a class of (partial) types closed under automorphisms. We say that $T$ is \textit{coordinatised} by $\mathcal{P}$ if for every $a\in\overline{\M}$ there is $n\in\omega$ and $a_i$ for $i\leq n$ such that $a_n=a$ and $\tp (a_i/a_{i-1})\in\mathcal{P}$ for all $i\leq n$, with $a_{-1}=\emptyset$. The sequence $(a_i:i\leq n)$ is called a \textit{coordinatising sequence}.
\item A type $q$ is said to be \textit{simple} if for each extension $p'\in S(B)$ there is a subset $A$ of $B$ with  $|A|\leq|T|$ such that $p'$ does not divide over $A$.
\item A type $q$ is said to be \textit{supersimple} if for each extension $p'\in S(B)$ there is a subset $A$ of $B$ with  $|A|<\aleph_0$ such that $p'$ does not divide over $A$.
\end{itemize}
\end{defn}	

\begin{propn}[\cite{HKP}, Proposition 4.2]\label{coordinatisation}
If $T$ is coordinatised by simple types then $T$ is simple. Furthermore if $T$ is coordinatised by supersimple types then $T$ is supersimple.
\end{propn}

We turn now to prove the following. We will work in $\underset{p\in\Prime}{\prod} \mathbb{Q}_p \diagup \mathcal{U}$.

\begin{propn}\label{simple}
Fix $b \in \Z ^+$ and let $\mathbf{R}$ be the ultraproduct $\underset{p\in\Prime}{\prod} \Z / p^b \Z \diagup \mathcal{U}$, where $\mathcal{U}$ is a non-principal ultrafilter on $\Prime$. Then $\mathbf{R}$ has supersimple theory.
\end{propn}
\begin{proof}

Since $\prod \pint / p^b \pint \diagup \mathcal{U} \iso \prod \Z / p^b\Z\diagup \mathcal{U}$, we will think about $\mathbf{R}$ inside the valued field structure $\mathbf{Q}= (\underset{p\in\Prime}{\prod} \mathbb{Q}_p\diagup \mathcal{U}, \Gamma, \prod \mathbb{F}_p\diagup\mathcal{U}, v,ac; \bar{\phi})$ where $\overline{\phi}$ is the definable function on $\underset{p\in\Prime}{\prod} \mathbb{Q}_p\diagup \mathcal{U}$ with $\overline{\phi}([x_p]_\mathcal{U}) = [p x_p]_\mathcal{U}$. Thus we want to show $\mathbf{R'} := \underset{p\in\Prime}{\prod} \pint / p^b \pint \diagup \mathcal{U} $ has supersimple theory with the induced structure in which the 0-definable relations on $\mathbf{R'}$ are all those arising from 0-definable relations on $\mathbf{Q}$.

In order to show that $\mathbf{R'}$ has a supersimple theory, by Proposition \ref{coordinatisation} it suffices to coordinatise it through supersimple types.
We will consider the class of types $\mathcal{P}$ of the form $\mathrm{tp}([p^ia_p + (p^b)] / [p^{i+1}a_p + (p^b)])$, for  $ [a_p + (p^b)] \in \mathbf{R'}$.

For any given $[a_p+ (p^b)]\in\mathbf{R'}$ we can have the coordinatising sequence $\alpha_0:=[p^{b-1}a_p + (p^b)] , \ldots,\alpha_{b-(i+1)}:=[p^{i}a_p+ (p^b)],\ldots , \alpha_{b-2}:=[p a_p + (p^b)], \alpha_{b-1}:=[a_p + (p^b)]$. Furthermore $\mathrm{tp}(\alpha_{b-i-1}/\alpha_{b-i-2}) = \mathrm{tp}([p^ia_p + (p^b)] / [p^{i+1}a_p + (p^b)])\subseteq\{x\in \mathbf{R'}: [p]x=[p^{i+1}a_p + (p^b)] \}$ which equals to a definable set, namely $S_{b-(i+1)} = [p^i a_p + p^{b-1}\mathbf{R'} + (p^b)] :=\{ [p^i a_p +p^{b-1}x + (p^b)] : x\in \mathbf{R'} \}$.  

Also there is a definable bijection $\vphi_{b-i-1}$ over $[p^i a_p + (p^b)]$ between $\prod \Z_p/p\Z_p \diagup \mathcal{U} $ and $S_{b-(i+1)}$. To see this, note that for $c_p \in \Z_p / p^b \Z_p$ written as $c_p:= c_{p(0)} +p c_{p(1)}+\ldots + p^{b-1} c_{p(b-1)} + (p^b) $ we can define, uniformly on $p$, a bijection $\vphi_{b-i-1,p}$ from $\Z_p / p \Z_p$ to $S_{b-i-1,p}:=\{z\in \Z_p /p^b\Z_p : p z= p^{i+1} c_p + (p^b) \}= p^i c_p + p^{b-1}\Z_p/p^b\Z_p + (p^b)$ as follows $\vphi_{b-i-1,p}(x):= p^{b-1} x + p^i c_p + (p^b) $.
Therefore we can find the wanted definable bijection $\vphi_{b-i-1}$ as the induced by $\vphi_{b-i-1,p}$ sending $[x_p]\in\prod \Z_p/p\Z_p \diagup \mathcal{U} $ to $[p^i a_p + p^{b-1} x_p + (p^b)]$.

Moreover $\prod \Z_p/p\Z_p \diagup \mathcal{U}$ is a pseudofinite field stably embedded in $\mathbf{Q}$, so all the types realised in $\prod \Z_p/p\Z_p \diagup \mathcal{U}$ are supersimple. Hence all the types of the form $\mathrm{tp}(\alpha_i/\alpha_{i-1})$ are also supersimple.

We note that $\mathcal{P}$ is closed under automorphisms of $\mathbf{R}'$.

Finally, note that for an element $[a_p+ (p^b)]\in \mathbf{R'}$ the sequence $\alpha_0= [p^{b-1} a_p + (p^b)]$, $\alpha_1= [p^{b-2}a_p + (p^b)]$, $\ldots$, $\alpha_{b-1}=[a_p + (p^b)]$ is a coordinatising sequence and we have that $\mathrm{tp}(\alpha_i/\alpha_{i-1})$ is in  $\mathcal{P}$ for $0\leq i<b-1$. 

Now we can apply Proposition \ref{coordinatisation} to conclude that since $\mathbf{R'}$ is coordinatised by supersimple types then it is supersimple.

\end{proof}

\begin{rem}
It is noted in Remark 4.3 of \cite{HKP} that if $(a_i: i\leq n)$ is a coordinatising sequence then $\mathrm{SU}(a_n)\leq \mathrm{SU}(a_n/ a_{n-1}) \oplus \ldots \oplus \mathrm{SU}(a_0)$. Hence the structure $\underset{p\in\Prime}{\prod} \Z / p^b \Z \diagup \mathcal{U}$ has finite $\mathrm{SU}$-rank and this rank is at most $b$ since the coordinatisatising sequence used in the proof of Proposition \ref{simple} has length $b$ and each of the types of the sequence has $\mathrm{SU}$-rank $1$, cf. \cite{HRU}. 
Furthermore we have that $\underset{p\in \Prime}{\prod} p^{b-1}\Z_p/p^b\Z_p \diagup \mathcal{U} \unlhd \underset{p\in\Prime}{\prod}\Z_p /p^b\Z_p \diagup \mathcal{U}$ has infinitely many cosets and also we have that $\underset{p\in\Prime}{\prod} ((\Z_p/p^b\Z_p) / (p^{b-1}\Z_p/p^b\Z_p ))\diagup \mathcal{U} \iso \underset{p\in\Prime}{\prod} \Z_p / p^{b-1}\Z_p \diagup \mathcal{U} $. Similarly the ideal $\underset{p\in\Prime}{\prod} p^{b-(i+1)}\Z_p/p^{b-i}\Z \diagup \mathcal{U}$ of $\underset{p\in\Prime}{\prod} \Z_p/p^{b-i}\Z_p \diagup \mathcal{U}$ has infinitely many infinite cosets. This last argument translates into a (dividing)forking sequence of types of length $b$ which means that $\mathrm{SU}(a_n)\geq b$.
Hence for any given $a\in\underset{p\in\Prime}{\prod} \pint / p^b \pint \diagup \mathcal{U}$ we have that $\mathrm{SU}(a)=b$.
\end{rem}

We can now use Proposition \ref{simple} together with Lemma \ref{CRT} to cover the following more general case.

\begin{cor}\label{simplecase}
Consider an ultrafilter on $\mathbb{N}$ such that there exists $b\in\mathbb{N}$ and $ U\in \mathcal{U}$ such that if $n\in U$ then $n$ is a product of less than $b$ primes and if $p^l$ divides $n$ then $l<b$. Then $R':= \prod \Z/n\Z \diagup \mathcal{U}$ is supersimple, of finite $\mathrm{SU}$-rank.
\end{cor}

First note that for example the ultrafilter $\mathcal{U}$ might include the set $$\{n\in\mathbb{N}:\ If\ p^e \mid n\ then\ e\leq b ,\ and\ n\ has\ less\ than\ b\ prime\ divisors\}.$$

\begin{proof}
Let $U\in\mathcal{U}$ be as in the hypothesis, and consider $\mathcal{V}$ the induced ultrafilter on $U$ by $\mathcal{U}$. We put $R:=\underset{n\in U}{\prod} \Z/n\Z \diagup \mathcal{V}$, so $R\iso R'$.
There is $V\in\mathcal{V}$ such that every $n\in V$ has the same number $d$ of prime factors, $p_{n(1)},p_{n(2)},\cdots p_{n(d)}$ and every $p_{n(i)}$ has the same exponent $e_i$. Hence using Lemma \ref{CRT} on $\mathcal{W}$ the induced ultrafilter on $V$ by $\mathcal{V}$ we have that $R\iso (\underset{m}{\prod} (\Z/p_{m(1)}^{e_1}\Z)\diagup \mathcal{W}_1)\times\cdots\times (\underset{m}{\prod}(\Z/p_{m(d)}^{e_d}\Z)\diagup\mathcal{W}_d)$. Since each one of $\underset{m}{\prod} (\Z/p_{m(i)}^{e_i}\Z)\diagup\mathcal{W}_i$ is supersimple by Proposition \ref{simple}, $R'$ is supersimple.

\end{proof}

\begin{rem}
The ring $\underset{m}{\prod}\Z/p_{m(1)}^{e_1}\Z\diagup \mathcal{W}_1 \times\ldots\times \underset{m}{\prod}\Z/p_{m(d)}^{e_d}\Z\diagup \mathcal{W}_d$ has $\mathrm{SU}$-rank exactly $e_1+\ldots+e_d$.
\end{rem}

\begin{rem}
There is an alternative, maybe more direct, way of proving Proposition \ref{simple}. We present here a brief sketch of the proof for $R:=\underset{p}{\prod} \Z_p/p^2\Z_p \diagup \mathcal{U}$, but the argument also yields Corollary \ref{simplecase}.

We note by the Ax-Kochen-Er{\v{s}}ov theorem that $\underset{p\in\Prime}{\prod}\Z/p^2\Z\diagup \mathcal{U}$ is elementary equivalent to $R':= \underset{p\in\Prime}{\prod}\mathbb{F}_p[[t]]/t^2\diagup \mathcal{U}$ in the language of rings, cf. Proposition 2.4.10 of \cite{vdD}. 

We have $R'$ is interpretable in $k:= \underset{p}{\prod} \mathbb{F}_p\diagup \mathcal{U}$. In order to see this we only need to note that for any prime $q$ the ring $\mathbb{F}_q[[t]]/t^2$ is uniformly interpretable in $\mathbb{F}_q$ since we can identify $a+bt +t^2\mathbb{F}_q[[t]] \in \mathbb{F}_q[[t]]/t^2$ with the pair $(a,b)$ and interpret addition, ($\oplus$), and multiplication, ($\ast$), from $\mathbb{F}_q[[t]]/t^2$  inside $\mathbb{F}_q \times \mathbb{F}_q$ in the following way.
\begin{itemize}
\item For pairs $(a,b),(c,d)$ we put $(a,b)\oplus (c,d):=(a+c, b+d)$;
\item For pairs $(a,b),(c,d)$ we put $(a,b)\ast (c,d)= (ac, ad+bc )$.
\end{itemize}

Since $k$ is supersimple of $\mathrm{SU}$-rank 1, $R'$ is supersimple of $\mathrm{SU}$-rank 2 and hence $R$ is also supersimple of $\mathrm{SU}$-rank 2.

Although $k$ is interpretable in $R'$ the isomorphism from $(k^2,\oplus,\ast)$ to $R'$ is not definable in $R'$. Otherwise, it would be also definable in $R$ and thus we would get that for some $U\in\mathcal{U}$ if $p\in\U$ then the ring $\Z_p/p^2\Z$ is isomorphic to the characteristic $p$ ring $((\Z/p\Z)^2, \oplus,\ast )$ but this is a contradiction.

In Proposition \ref{simple} we chose to give the proof using coordinatisation through supersimple types because we believe it provides added information for the class of finite rings of the form $\Z/p^n\Z $. In particular, it may prove useful towards showing that for a fixed $d$ the class of rings $\{ \Z/p^d\Z : p\in\Prime \}$ is a (weak) asymptotic class in the sense of \cite{Elw}, see also \cite{EM}. 
\end{rem}

\section{TP2 case}

Not every ultraproduct of finite residue rings is NTP2.
First, we note the following. The proof is routine and omitted here.

\begin{clm}\label{SLprodcommute}
If $\{R_i: i\in I\}$ is a collection of rings, then $SL_2(\underset{i\in I}{\prod}R_i/\mathcal{U})\cong \underset{i\in I}{\prod} (SL_2(R_i))/\mathcal{U}$.
\end{clm}

For an analogous result (over fields, but with arbitrary groups of Lie type) see Proposition 1 of \cite{POI}.

Now we will present a useful necessary condition for a theory to be NTP2, from \cite{CKS}.

\begin{lem}\label{lemma}
Let $T$ be $NTP2$, $G$ a definable group in $\M \models T$ and $(H_i)_{i\in\omega}$ a uniformly definable family of normal subgroups of $G$, with $H_i=\vphi(x,a_i)$. Let $H:=\underset{i\in\omega}{\bigcap} H_i$, and put $H_{\neq j}:=\underset{i\in \omega\setminus\{ j \}}{\bigcap} H_i$. Then there is some $i^*\in\omega$ such that, $[H_{\neq i^*}:H]$ is finite.
\end{lem}

\begin{propn}\label{TP2case}
Let $\mathcal{U}$ be an ultrafilter on $\mathbb{N}$ such that for every $b\in\mathbb{N}$ the set $$\{n\in\mathbb{N}:\ there\ are\ at\ least\ b\ prime\ divisors\ of\ n\}$$ is in $\mathcal{U}$. Then $\prod \Z/n\Z \diagup \mathcal{U}$ has TP2 theory.
\end{propn}
First we note these ultrafilters exist, since the collection $$\{\{n\in\mathbb{N}:\ there\ are\ at\ least\ b\ prime\ divisors\ of\ n\}: b\in \mathbb{N} \}$$ has the finite intersection property. 

\begin{proof}
Let $\mathcal{R} := \prod (\Z/n\Z) / \mathcal{U}$.
We want to find a definable group $G$ in $\mathcal{R}$ and a uniformly definable family of normal subgroups $\{H_i\}_{i<\omega}$ that contradicts the conclusion from Lemma \ref{lemma}, and  to do this we consider 
for each $j<\omega$ the group $$SL_2(\Z/j\Z)\iso G_j:=F_j\times SL_2\left( \Z/p_{j,1}^{e_{j,1}}\Z \right) \times \cdots \times SL_2\left( \Z/p_{j,b_j}^{e_{j,b_j}}\Z \right)$$ such that $F_j$ is isomorphic to $SL_2(\Z/ n\Z)$ where $n$ can only have prime factors smaller than  $p_{j,1}$, and both $b_i$ and $p_{i,1}$ increase as $i\rightarrow\infty$, in such a way that $p_{j,k}> p_{j',k'}$ whenever $j>j'$ or both $ j=j'$ and $k>k'$.

We can now for a given $j$ and $k$ with $1\leq k \leq b_j$ find non-central elements $A_{j,k}= \left( \begin{smallmatrix} a_{j,k} & 0 \\ 0 & b_{j,k} \end{smallmatrix} \right) \in SL_2\left( \Z/p_{j,k}^{e_{j,k}}\Z \right)$ with $a_{j,k}\neq b_{j,k}$. Consider $A_{j,k}$ inside $G_j$ occurring as the $k+1$-th entry in $$\overline{A_{j,k}}= \left( Id_{F_j},Id_{SL_2\left( \Z/p_{j,1}^{e_{j,1}}\Z \right)}, \ldots ,A_{j,k},\ldots, Id_{SL_2\left( \Z/p_{j,b_j}^{e_{j,b_j}}\Z \right)} \right)$$  

Now let $\mathcal{F}_{j,k}$ be the conjugacy class in $G_j$ of $\overline{A_{j,k}}$. This is uniformly definable (over $j$) using $\overline{A_{j,k}}$ as a parameter. Elements in $\mathcal{F}_{j,k}$ are of the form $(1, Id, \ldots, gA_{j,k}g^{-1},\ldots,Id)$ where $g\in SL_2\left( \Z/p_{j,k}^{e_{j,k}}\Z \right)$.

Consider now $N_{j,k}=C_{G_j}(\mathcal{F}_{j,k})$. Then this is a normal subgroup of $G_j$ since for every $\alpha \in G_j$ and every $\gamma\in N_{j,k}$ we have that $\alpha\mathcal{F}_{j,k}\alpha^{-1}=\mathcal{F}_{j,k} $ so $\alpha\gamma\alpha^{-1}\in N_{j,k}$. Furthermore $N_{j,k}=C_{G_j}(\mathcal{F}_{j,k})$ is uniformly definable using $\overline{A_{j,k}}$ through the formula defining the centralizer, namely $\vphi(x,\overline{A_{j,k}}):= \forall g\in G_j (x g \overline{A_{j,k}}g^{-1}x^{-1}=g \overline{A_{j,k}}g^{-1}) $. We have $$N_{j,k}= F_j \times SL_2\left( \Z/p_{j,1}^{e_{j,1}}\Z \right)\times\ldots\times C_{SL_2\left( \Z/p_{j,k}^{e_{j,k}}\Z \right)} (\mathcal{F}_{j,k})\times \ldots \times SL_2\left( \Z/p_{j,b_j}^{e_{j,b_j}}\Z \right).$$

Define $N_{j,\neq k}=\underset{l\neq k}{\bigcap} N_{j,l}$ and $N_j=\underset{l}{\bigcap} N_{j,l}$. Then the index $$[N_{j,\neq k}:N_j]=[SL_2\left( \Z/p_{j,k}^{e_{j,k}}\Z \right): C_{SL_2\left( \Z/p_{j,k}^{e_{j,k}}\Z \right)}(A_{j,k}^{SL_2\left( \Z/p_{j,k}^{e_{j,k}}\Z \right)})]$$ is at least $p_{j,k}$.
Indeed, if $[N_{j,\neq k}:N_j]=\lambda$ then for all $B\in SL_2\left( \Z/p_{j,k}^{e_{j,k}}\Z \right)$ we have that $B^\lambda\in C_{SL_2\left( \Z/p_{j,k}^{e_{j,k}}\Z \right)}(A_{j,k}^{SL_2\left( \Z/p_{j,k}^{e_{j,k}}\Z \right)})$ which means that in particular $B^\lambda A_{j,k} B^{-\lambda}=A_{j,k}$. Considering the matrix $D=\left(\begin{smallmatrix} 1 & 1\\ 0 & 1 \end{smallmatrix}\right)$, $D^\lambda A_{j,k} D^{-\lambda}= \left(\begin{smallmatrix} a_{j,k} & \lambda (b_{j,k}-a_{j,k}) \\ 0 & b_{j,k} \end{smallmatrix} \right)= A_{j,k}$ only if $\lambda (b_{j,k}-a_{j,k})$ is a multiple of $p_{j,k}^{e_{j,k}}$.
 Hence the index $[N_{j,\neq k}:N_j]$ is at least $p_{j,k}$. And we have $p_{j,k}\rightarrow \infty$ as $j\rightarrow \infty$.
 
We want to consider $H=\underset{j}{\prod} N_j / \mathcal{U} = \underset{j}{\prod}\left(\underset{l}{\bigcap} N_{j,l} \right)/ \mathcal{U} $ and $H_{\neq k}=\underset{j}{\prod}N_{j,\neq k} / \mathcal{U} = \underset{j}{\prod} \left(\underset{l\neq k}{\bigcap} N_{j,l} \right)/\mathcal{U}$ and we must show that the index $[H_{\neq k}:H]$ is infinite. But we know that $[N_{j,\neq k}:N_j]\geq p_{j,k}$ for every $j$. Also if $j'>j$ then $[N_{j',\neq k}:N_{j'}]\geq p_{j',k} > p_{j,k}$. By \L$ o\acute{s}$'s Theorem we have $\left[ \prod(N_{j,\neq k})/\mathcal{U} : \prod (N_j)/\mathcal{U} \right] \geq p_{j,k}$ for each $p_{j,k}$ 
hence in the ultraproduct the index $[H_{\neq k}:H]$ is infinite, for all $k$.
Hence we just need to show that $H=\underset{i}{\bigcap} H_i$ and $H_{\neq k}=\underset{i\neq k}{\bigcap} H_i $ for some family of uniformly definable normal subgroups $\{H_i : i<\omega \}$. Put $H_i= \underset{j}{\prod} N_{j,i}/\mathcal{U}$ and recall that $H=\underset{j}{\prod}\left(\underset{l}{\bigcap} N_{j,l} \right)/\mathcal{U}$ and $H_{\neq k}= \underset{j}{\prod} \left(\underset{l\neq k}{\bigcap} N_{j,l} \right)/\mathcal{U}$ and that $\underset{j}{\prod}\left(\underset{l}{\bigcap} N_{j,l} \right)/\mathcal{U} = \underset{l}{\bigcap}\left(\underset{j}{\prod} N_{j,l}/\mathcal{U} \right)$ and $\underset{j}{\prod} \left(\underset{l\neq k}{\bigcap} N_{j,l} \right)/\mathcal{U}=\underset{l\neq k}{\bigcap} \left(\underset{j}{\prod} N_{j,l} /\mathcal{U}\right)$. Thus, by Lemma \ref{lemma} applied to the family of groups $\{ \underset{j}{\prod} N_{j,i}\}_{i<\omega}$ the ring $\prod \Z/n\Z /\mathcal{U}$ has the tree property of the second kind.

\end{proof}

\bibliographystyle{abbrv} 
\bibliography{ReferencesTR}

\end{document}